\documentclass[11pt]{article}
\usepackage{enumerate}
\usepackage{amssymb,a4wide,latexsym,makeidx,epsfig,fleqn}
\usepackage{amsthm}
\usepackage{amsmath}
\usepackage{enumerate}
\usepackage{extarrows}
\usepackage{graphicx}
\usepackage{subfigure}
\usepackage{float}
\usepackage[justification=centering]{caption}
\usepackage{tikz} 
\usepackage{amsmath}
\usepackage{caption}
\usepackage{mathrsfs}
\usetikzlibrary{positioning}
\usetikzlibrary{decorations.pathreplacing}
\usepackage[numbers,sort&compress]{natbib} 
\newtheorem{theorem}{Theorem}[section]

\newtheorem{lemma}[theorem]{Lemma}

\newtheorem{corollary}[theorem]{Corollary}
\newtheorem{problem}[theorem]{Problem}
\newtheorem{conjecture}[theorem]{Conjecture}

\begin{document}
\textwidth 150mm \textheight 225mm
\title{Spectral extremal problems for fractional $ID$-$[a,b]$-factor-critical graphs
\thanks{Supported by the National Natural Science Foundation of China (Nos. 12271439 and 12001434)}}
\author{Zengzhao Xu$^{a,b}$, Ligong Wang$^{a,b,}$\footnote{Corresponding author.}, Weige Xi$^{c}$\\
{\small $^{a}$ School of Mathematics and Statistics, Northwestern Polytechnical University,}\\
{\small  Xi'an, Shaanxi 710129, P.R. China.}\\
{\small $^b$ Xi'an-Budapest Joint Research Center for Combinatorics, Northwestern
	Polytechnical University,}\\
{\small  Xi'an, Shaanxi 710129, P.R. China.}\\
{\small $^{c}$ College of Science, Northwest A\&F University, Yangling, Shaanxi 712100, China}\\
{\small E-mail: xuzz0130@163.com; lgwangmath@163.com; xiyanxwg@163.com}\\}
\date{}
\maketitle
\begin{center}
\begin{minipage}{120mm}
\vskip 0.3cm
\begin{center}
{\small {\bf Abstract}}
\end{center}
{\small    A factor of a graph is essentially a specific type spanning subgraph. In recent years, the spectral extremal problem of characterizing the existence of graph factors via eigenvalues has been widely studied. This paper focuses on fractional $ID$-$[a, b]$-factor-critical graphs, which are a natural generalization of fractional $[a,b]$-factors. Let $r \ge 1$ be an integer. A graph $G$ is fractional $ID$-$[a, b]$-factor-critical if for every independent set $I$ of $G$ with $|I| = r$, $G - I$ has a fractional $[a, b]$-factor. In 2026, Jia, Fan and Liu posed the spectral version conjecture for a graph to be fractional $ID$-$[a, b]$-factor-critical [Linear Algebra Appl. 732 (2026) 1-17]. In this paper,
	we first prove the conjecture holds for connected graphs when $b\ge 2r+2$. Furthermore, for minimum degree $\delta(G)\ge a+r$, we present spectral radius and size conditions that ensure a graph is fractional $ID$-$[a, b]$-factor-critical, which improve the results of Jia, Fan and Liu.

}
\end{minipage}
\end{center}

\section{Introduction } 
\quad Let $G$ be a simple, finite and undirected graph with vertex set $V(G)$ and edge set $E(G)$, respectively. We use $|V (G)| =n$ 
and $|E(G)| =e (G)$ to denote the order and the size of $G$. The complement $\overline{G}$ of $G$ 
is the graph with $V(\overline{G}) = V (G)$ and two distinct vertices in $G$ are adjacent if and only if they 
are non-adjacent in $G$. For a vertex $v_i \in V(G)$,  we use $N_G(v_i)$ to denote its neighborhood, and we set
 $d_G(v_i) = |N_G(v_i)|$ as its degree. In addition, let $N_G[v_i]=N_G(v_i)\cup \{v_i\}$. For a vertex $v_i \in V(G)$ and a vertex subset $S \subseteq V(G)$, let $N_S(v_i) = N_G(v_i) \cap S$ and $d_S(v_i) = |N_S(v_i)|$. We use $K_n$ to denote the complete graph of order $n$. Let $\delta(G)$ be the minimum degree of a graph $G$.  For $U \subseteq V(G)$, we use $G[U]$ to denote the subgraphs of $G$ induced by $U$. For $S_1,S_2\subseteq V(G)$ with  $S_1 \cap S_2 = \emptyset$, let
  $E_G(S_1, S_2)$ be the set of edges in $G$ with one endpoint in $S_1$ and the other in $S_2$, and let $e_G(S_1, S_2) = |E_G(S_1, S_2)|$. Let $G_1 \cup G_2$ be the disjoint union of graphs $G_1$ and $G_2$. The join $ G_1 \vee G_2$ is obtained from $G_1 \cup G_2$ by joining all possible edges between $ V(G_1)$ and $V(G_2)$. For further details and related concepts, we refer readers to \cite{BM}.

For a graph $G$ with vertex set $V(G)$ and edge set $E(G)$, the adjacency matrix of $G$ is defined as the matrix $A(G) = (a_{ij})_{n \times n}$ with $a_{ij} = 1 $ if $v_iv_j\in E(G)$ and $a_{ij}=0$ otherwise. The largest eigenvalue of $A(G)$, denoted by $\lambda(G)$, is called the spectral radius of a graph $G$.

 A factor of a graph is essentially a spanning subgraph, and related research can be traced back to the pioneering work of the Danish mathematician Petersen in 1891. Factor theory is an important branch of graph theory. It focuses on the decomposition of a graph into subgraphs with given structural properties. This theory provides a useful framework for problems such as matching theory and network design. In addition, factor theory plays a significant role in several applied disciplines, including computer science, operations research, and chemical graph theory. Please refer to the monographs \cite{AK,YL} for further
 details.
 
  Let $h \colon E(G) \to [0, 1]$ be
a function defined on the edge set $E(G)$ and let $b\ge a$ be positive integers. For every $v \in V(G)$, if $a \leq \sum_{e \in E_G(v)} h(e) \leq b$, then the spanning subgraph 
with edge set $E_h = \{e \in E(G) \mid h(e) > 0\}$, denoted by $G[E_h]$, is called a fractional $[a, b]$-factor of $G$ with indicator function $h$.

The study of graph factors through eigenvalues has an abundant history. Brouwer and Haemers \cite{BH} characterized the condition for a regular graph to contain a perfect matching in terms of its third-largest eigenvalue. This work is one of the early publications that explored factor existence using eigenvalue methods. Subsequently, various improvements and extensions to this result have been made by researchers (\cite{C,CG,CGH}). O \cite{O3} extended the result of Brouwer and Haemers \cite{BH} to graphs that are not necessarily
regular, and presented a spectral radius condition for a connected graph to contain a 1-factor. Li, Fan and Zhu \cite{LFZ} investigated the existence of fractional $[a, b]$-factors in graphs
from the perspectives of size and  spectral radius, respectively. For more details, see survey \cite{FLLO} and \cite{FLZ,FL,FLL,FLA,G,HL,L2,LFZ,MW,O,O3,O2}.

This paper focuses on fractional $ID$-$[a, b]$-factor-critical graphs, which are a natural generalization of fractional $[a,b]$-factors. Let $r \ge 1$ be an integer. A graph $G$ is fractional $ID$-$[a, b]$-factor-critical if for every independent set $I$ of $G$ with $|I| = r$, $G - I$ has a fractional $[a, b]$-factor. This concept was first introduced by Zhou, Sun and Liu \cite{ZSL}. Since then, several structural conditions have been established to guarantee that a graph is fractional $ID$-$[a, b]$-factor-critical, such as eigenvalue conditions \cite{JFL,MW}, neighborhood conditions \cite{YH,ZY}, and degree conditions \cite{ZSL}.

In 2026, Jia, Fan and Liu \cite{JFL} proposed a problem on characterizing fractional $ID$-$[a, b]$-factor-critical graphs.
\begin{problem}(\cite{JFL})
	What is a sufficient condition in terms of the spectral radius to guarantee the existence of a fractional $ID$-$[a, b]$-factor-critical graph?
\end{problem}
Jia, Fan and Liu \cite{JFL} characterized the size and spectral radius conditions for fractional $ID$-$[a, b]$-factor-critical graphs and posed the following conjecture at the end of the paper.
\begin{theorem}(\cite{JFL})\label{T1.1}
	Let $a$ and $b$ be two positive integers with $a \leq b$. Let $G$ be a connected graph of order $n \geq \frac{75}{2}b^2 + (35a + 5r - \frac{37}{2})b - r^2 + (2a - \frac{3}{2})r + 8a^2 - 8a + 3$ and minimum degree $\delta(G) \geq a + r$, where $b \geq 2r$. If
$e(G) > \frac{1}{2}n^2 - (r + \frac{1}{2})n + r^2 + \frac{(a+1)}{2}r,$
	then $G$ is fractional $ID$-$[a,b]$-factor critical.
\end{theorem}
\begin{theorem}(\cite{JFL})\label{T1.2}
	Let $a$, $b$ be two positive integers with $a \leq b$. Let $G$ be a connected graph of order $n \geq \frac{75}{2}b^2 + (35a + 5r - \frac{37}{2})b - r^2 + (2a - \frac{3}{2})r + 8a^2 - 8a + 3$ and minimum degree $\delta(G) \geq a + r$, where $b \geq 2r$. If
	$\lambda(G) > n - r - 1,$
	then $G$ is fractional $ID$-$[a,b]$-factor critical.
\end{theorem}
\begin{conjecture}\label{c1}(\cite{JFL})
	Let $a, b$ be two positive integers with $a \leq b$, and let $G$ be a graph of order $n \geq a + r + 1$, where $r \geq 1$ is a given integer. If
	$\lambda(G) \geq \lambda\left(I_r \vee \left(K_{a-1} \vee \left(K_{n-a-r} \cup K_1\right)\right)\right),$
	then $G$ is fractional $ID$-$[a,b]$-factor critical unless $G \cong I_r \vee \left(K_{a-1} \vee \left(K_{n-a-r} \cup K_1\right)\right)$.
\end{conjecture}
Motivated by \cite{JFL}, we study spectral radius and size conditions for a graph to be fractional $ID$-$[a,b]$-factor critical. We first introduce some important theorems. 
In \cite{A}, Anstee characterized the necessary and sufficient conditions for the existence of a fractional $[a,b]$-factor in a graph. The following lemma can be derived from \cite{A}.
\begin{theorem}(\cite{A})\label{T}
	Let $G$ be a graph and let $a, b$ be two positive integers with $a \leq b$. Then $G$ has a fractional $[a,b]$-factor if and only if for any subset $S \subseteq V(G)$, we have
	$b|S| - a|T| + \sum_{v \in T} d_{G-S}(v) \geq 0,$
	where $T = \{v \mid v \in V(G) - S \text{ and } d_{G-S}(v) < a\}$.
\end{theorem}
The following corollary is a direct consequence of Theorem \ref{T}.
\begin{corollary}\label{co1}(\cite{JFL})
	Let $G$ be a graph and let $a, b$ be two positive integers with $a \leq b$. Then $G$ is fractional $ID$-$[a,b]$-factor critical if and only if for any independent set $I$ and subset $S \subseteq V(G - I)$, we have
$	b|S| - a|T| + \sum_{v \in T} d_{G-I-S}(v) \geq 0,
$ where $T = \{v \mid v \in V(G) \setminus (S\cup I) \text{ and } d_{G-I-S}(v) \le a-1\}$.
\end{corollary}
Let $H_{n}^{a,r}=rK_1\vee (K_{a-1}\vee (K_{n-a-r}\cup K_1)).$ For Conjecture \ref{c1}, we prove the following.
\noindent\begin{theorem}\label{T1}  \  Let $a, r \ge 1$ and $b \ge \max\{a, 2r+2\}$ be integers. Let $G$ be a connected graph of order $n \geq (r+1)(b-r)+2a^2+2a(r+2)-1$. If
	$\lambda(G) \geq \lambda(H_{n}^{a,r}),$
	then $G$ is fractional $ID$-$[a, b]$-factor critical unless $G \cong H_{n}^{a,r}$.
\end{theorem}

Let $F_{n}^{a,b,r}$ be a graph obtained from
$rK_1\vee (K_{a} \vee ( K_{n-(a+b+r+1)} \cup (b+1)K_1 ))$ by adding $a-1$ edges between one vertex in $V((b+1)K_1)$ and $a-1$ vertices in $V(K_{n-(a+b+r+1)})$, where $F_{n}^{a,b,r}$ is shown in Figure 1.
For $\delta(G)\ge a+r$, we obtain some results, which improve the results of \cite{JFL}.
\begin{figure}[htbp]
	\centering
	\begin{minipage}[c]{0.45\textwidth}
		\centering
		\begin{tikzpicture}[scale =1.2]
			\node[circle,fill=black,draw=black,inner sep=2.2pt] (v1) at (-0.2,0) {};
			\node[circle,fill=black,draw=black,inner sep=2.2pt] (v6) at (0.8,0) {};
			
			\node[circle,fill=black,draw=black,inner sep=2.2pt] (v1.1) at (-1.6,0) {};
			\node[circle,fill=black,draw=black,inner sep=2.2pt] (v1.2) at (-2.6,0) {};
			
			\node[circle,fill=black,draw=black,inner sep=2.2pt] (v2) at (-0.2,2.0) {};
			\node[circle,fill=black,draw=black,inner sep=2.2pt] (v5) at (0.8,2.0) {};
			
			\node[circle,fill=black,draw=black,inner sep=2.2pt] (v7) at (-1.6,2.0) {};
			\node[circle,fill=black,draw=black,inner sep=2.2pt] (v8) at (-2.6,2.0) {};
			\node[circle,fill=black,draw=black,inner sep=1pt] (v9) at (-2.1,2.0)  {};
			
			\node (v10) at (-2.1,-0.3) {};
			\node (v11) at (1.0,2.4) {};
			\node (v13) at (-2.6,2.4) {};
			\node (v14) at (0.5,2.) {};
			\node (v15) at (0.3,2.1) {};
			\node (v16) at (0.1,2.1) {};
			\node (v17) at (-0.8,2.8) {};
			\node (v18) at (0.5,-0.3) {};
			
			\node (v9) at (-2.1,2.0) {};
			\node (v20) at (0.5,2.0) {};
			\node (v21) at (-2.1,0) {};
			\node (v22) at (0.5,0) {};

			\node[below] at (v1) { };
			\node[left] at (v2) { };
			\node[left] at (v5) { };
			\node[below] at (v6) { };
			\node[right] at (v7) { };
			\node[right] at (v8) { };
			\node[below] at (v10) {\large$I=rK_{1}$};  
			\node[below] at (v18) {\large$K_{a}$};  
			\node[overlay] at (v11) {\large$K_{n-(a+r+b+1)}$};   
			\node[overlay] at (v13) {\large$(b+1)K_1$};  
			\node[overlay] at (v17) {\large$a-1$};   
			\draw [line width=1.2pt](v9) -- (v21);
			\draw [line width=1.2pt](v9) -- (v22);
			\draw [line width=1.2pt](v20) -- (v21);
			\draw[line width=1.2pt] (v20) -- (v22);
			\draw[line width=1.2pt] (v21) to[bend right=30] (v22);  
			\draw [line width=1.2pt](0.3,2.0) ellipse (0.65cm and 0.17cm);
			\draw[line width=1.2pt] (v9) to[bend left=40] (v14);  
			\draw[line width=1.2pt] (v9) to[bend left=35] (v15);  
			\draw[line width=1.2pt] (v9) to[bend left=30] (v16);  
			
			\fill (0.1,0) circle (1pt); 
			\fill (0.3,0) circle (1pt); 
			\fill (0.5,0) circle (1pt); 
			\draw [line width=1.2pt](0.3,0) ellipse (0.65cm and 0.17cm);
			
			\fill (-1.9,0) circle (1pt); 
			\fill (-2.1,0) circle (1pt); 
			\fill (-2.3,0) circle (1pt); 

			\fill (0.1,2.0) circle (1pt); 
			\fill (0.3,2.0) circle (1pt); 
			\fill (0.5,2.0) circle (1pt); 
			\fill (-1.9,2.0) circle (1pt); 
			
			\fill (-2.3,2.0) circle (1pt); 
		\end{tikzpicture}
		\caption{The graph $F_{n}^{a,b,r}$}
	\end{minipage}
	\hfill
	\begin{minipage}[c]{0.45\textwidth}
		\centering
		\begin{tikzpicture}[scale =1.2]
			\node[circle,fill=black,draw=black,inner sep=2.2pt] (v1) at (-0.2,0) {};
			\node[circle,fill=black,draw=black,inner sep=2.2pt] (v6) at (0.8,0) {};
			
			\node[circle,fill=black,draw=black,inner sep=2.2pt] (v1.1) at (-1.6,0) {};
			\node[circle,fill=black,draw=black,inner sep=2.2pt] (v1.2) at (-2.6,0) {};
			
			\node[circle,fill=black,draw=black,inner sep=2.2pt] (v2) at (-0.2,2.0) {};
			\node[circle,fill=black,draw=black,inner sep=2.2pt] (v5) at (0.8,2.0) {};
			
			\node[circle,fill=black,draw=black,inner sep=2.2pt] (v7) at (-1.6,2.0) {};
			\node[circle,fill=black,draw=black,inner sep=2.2pt] (v8) at (-2.6,2.0) {};
			\node[circle,fill=black,draw=black,inner sep=1pt] (v9) at (-2.1,2.0)  {};
			
			\node (v10) at (-2.1,-0.3) {};
			\node (v11) at (1.0,2.8) {};
			\node (v13) at (-2.6,2.8) {};
			\node (v14) at (0.5,2.) {};
			\node (v15) at (0.3,2.1) {};
			\node (v16) at (0.1,2.1) {};
			\node (v17) at (-0.8,3.0) {};
			\node (v18) at (0.5,-0.3) {};
			
		\node (v31) at (0.5,2.2) {};
		\node (v32) at (0.3,2.2) {};
		\node (v33) at (0.1,2.2) {};

		\node (v34) at (-1.9,2.2)  {};
		\node (v35) at (-2.1,2.2) {};
		\node (v36) at (-2.3,2.2) {};
			
			\node (v9) at (-2.1,2.0) {};
			\node (v20) at (0.5,2.0) {};
			\node (v21) at (-2.1,0) {};
			\node (v22) at (0.5,0) {};

			\node[below] at (v1) { };
			\node[left] at (v2) { };
			\node[left] at (v5) { };
			\node[below] at (v6) { };
			\node[right] at (v7) { };
			\node[right] at (v8) { };
			\node[below] at (v10) {\large$I=rK_{1}$};  
			\node[below] at (v18) {\large$K_{a}$};  
			\node[overlay] at (v11) {\large$K_{n-(a+r+b+1)}$};   
			\node[overlay] at (v13) {\large$(b+1)K_1$};  
			\node[overlay] at (v17) {\large$a-1$};   
			\draw [line width=1.2pt](v9) -- (v21);
			\draw [line width=1.2pt](v9) -- (v22);
			\draw [line width=1.2pt](v20) -- (v21);
			\draw[line width=1.2pt] (v20) -- (v22);
			\draw[line width=1.2pt] (v21) to[bend right=30] (v22);  
			\draw [line width=1.2pt](0.3,2.0) ellipse (0.65cm and 0.17cm);
			\draw[line width=1.2pt] (v33) to[bend right=30] (v34);  
			\draw[line width=1.2pt] (v32) to[bend right=35] (v35);  
			\draw[line width=1.2pt] (v31) to[bend right=40] (v36);  
			
			\fill (0.1,0) circle (1pt); 
			\fill (0.3,0) circle (1pt); 
			\fill (0.5,0) circle (1pt); 
			\draw [line width=1.2pt](0.3,0) ellipse (0.65cm and 0.17cm);
			
			\fill (-1.9,0) circle (1pt); 
			\fill (-2.1,0) circle (1pt); 
			\fill (-2.3,0) circle (1pt); 

			\fill (0.1,2.0) circle (1pt); 
			\fill (0.3,2.0) circle (1pt); 
			\fill (0.5,2.0) circle (1pt); 
			\fill (-1.9,2.0) circle (1pt); 
			
			\fill (-2.3,2.0) circle (1pt); 
		\end{tikzpicture}
		\caption{The graphs $ \mathscr{F}_{n}^{a,b,r}$}
	\end{minipage}
	\label{figure:1}
\end{figure}
\noindent\begin{theorem}\label{T2} 
\  For integers $a,b,r$ with $r\ge1$ and $b\ge a\ge r+2$,  let $G$ be a connected graph of order $n \geq 2(4b+a+r + 2)(b+r + 2)$  with minimum degree $\delta(G)\ge a+r$. If
$\lambda(G) \geq \lambda(F_{n}^{a,b,r}),$
then $G$ is fractional $ID$-$[a, b]$-factor critical, unless $G \cong F_{n}^{a,b,r}$.
\end{theorem}
\noindent{\bf Remark.} By Cliam 1 in Lemma \ref{le:2.2} in Section 4, we obtain that $\lambda(F_{n}^{a,b,r})<n-b-1\le n-r-1$ for $b\ge r$. Hence, Theorem \ref{T2} improves Theorem \ref{T1.2}.
\noindent\begin{theorem}\label{T3}  \  For integers $a,b,r$ with $r\ge1$ and $b\ge a\ge 1$, let $G$ be a connected graph of order $n\ge4(a+b+r+1)$ with minimum degree $\delta(G)\ge a+r$. If
	$e(G)\ge \binom{n-r-b-1}{2}+r(n-r)+a(b+1)+a,$
then $G$ is fractional $ID$-$[a, b]$-factor critical.
\end{theorem}
\noindent{\bf Remark.} It is easy to verify that, in most cases, $$\frac{1}{2}n^2 - (r + \frac{1}{2})n + r^2 + \frac{(a+1)}{2}r>\binom{n-r-b-1}{2}+r(n-r)+a(b+1)+a.$$ Hence Theorem \ref{T3} is better than Theorem \ref{T1.1}. In fact, The result of Theorem \ref{T3}  achieves the best possible condition. By direct calculation,  $e(F_{n}^{a,b,r})=\binom{n-r-b-1}{2}+r(n-r)+a(b+1)+a-1$. Since $F_{n}^{a,b,r}$ is not fractional $ID$-$[a, b]$-factor critical due to Lemma \ref{le:1}, hence the condition in Theorem \ref{T3} is best possible.

The remainder of this paper is organized as follows. In Section 2, we introduce some lemmas for the proofs of subsequent theorems. In Sections 3, 4, and 5, we prove Theorems \ref{T1}, \ref{T2}, and \ref{T3}, respectively.
\section{Preliminaries}

\quad\quad In this section, we presents essential lemmas for the proofs of subsequent theorems. 

\begin{lemma}\label{le:4}
	Let $b\ge a\ge 1$ and $r\ge 1$ be integers, and let $G$ be a connected graph of order $n$ with minimum degree $\delta(G)\ge a+r$. If there exist two disjoint subsets $S, T$ and independent set $I$ of $V (G)$ with $|S| = s$, $|T|=t$ and $|I|=r$ such that
	$\sum_{v \in T} d_{G-I-S}(v) \leq at - bs  - 1,	$
	then
	\begin{enumerate}
		\item[(i)] $1 \leq s \leq t  - 1, \quad t \geq b + 1.$
		\item[(ii)] $e(G)\le at-bs-1 + st + \frac{(n -r- t)(n -r- t - 1)}{2}+r(n-r)$.
	\end{enumerate}
\end{lemma}
\begin{proof}
(i)	If $s=0$, since $\delta(G) \geq a+r$, we obtain $d_{G-I-S}(v) =d_{G-I}(v) \geq \delta(G)-r\ge a$ for any $v \in T$ and
	$
	at \leq \sum_{v \in T} d_{G-I}(v) \leq at-1,
	$
	a contradiction. Hence $s \geq 1$.
	
	If $s \geq t$, note that $b\ge a\ge 1$, we have
	$
	0\le \sum_{v \in T} d_{G-I-S}(v) \leq at - bs  - 1\le (a-b)t-1\le -1,
	$
	a contradiction. Hence, $s \leq t-1$. 
	
	Since $\delta(G) \geq a+r$ and $d_{G-I-S}(v) \geq a+r-r- s=a-s$ for any $v \in T$, we obtain
	$
	(a-s)t \leq \sum_{v \in T} d_{G-I-S}(v) \leq at - bs  - 1.
	$
	Therefore, $t \geq b + \frac{1}{s}$. Since $t$ is a positive integer and $s \geq 1$, we obtain $t \geq b + 1$.
	
	(ii) Let $W = V(G)\setminus(I\cup S \cup T).$ Then
	\begin{equation*}
		\begin{aligned}
			e(G)&=e(W, T)+e(T) +e(S, T) + e(G - T-I)+e(I,G-I)\\
			&\leq \sum_{v \in T} d_{G-I-S}(v) + e(S, T)  + e(G - T-I)+e(I,G-I) \\
			&\leq at-bs-1 + st + \binom{n -r- t}{2}+r(n-r) \\
			&= at-bs-1 + st + \frac{(n -r- t)(n -r- t - 1)}{2}+r(n-r).
		\end{aligned}
	\end{equation*}
	This completes the proof of Lemma \ref{le:4}.
\end{proof}

\begin{lemma}\label{le:1}
	Let $a, b, r$ be integers such that $1 \le a \le b$ and $r \ge 1$. Then	
	\begin{enumerate}
		\item[(i)] $H_{n}^{a,r}$ is not fractional $ID$-$[a,b]$-factor critical.
		\item[(ii)] $F_{n}^{a,b,r}$ (see Figure 1) is not fractional $ID$-$[a,b]$-factor critical.
	\end{enumerate}
\end{lemma}
\begin{proof}
	Recall that $H_{n}^{a,r}=rK_1\vee \big(K_{a-1}\vee (K_{n-a-r}\cup K_1)\big).$ For $v\in K_1$, we obtain $d_G(v)=a+r-1$. Let $I=rK_1$, then $d_{G-I}(v)=a-1$. Let $h: E(G) \to [0,1]$ be a function.
	Then for vertex $v \in V(K_1)$,  $\sum_{e \in E_{G-I}(v)} h(e) \leq a-1$, which implies that $G-I$ has no fractional $[a,b]$-factor. Hence $H_{n}^{a,r}$ is not fractional $ID$-$[a,b]$-factor critical.
	
	 Let $S = V(K_{a})$, $I=rK_1$ and $T = V((b + 1)K_1)$ in $F_{n}^{a,b,r}$. Note that $\sum_{v \in T} d_{G-I-S}(v) = a - 1$, we have
	$b|S|-a|T| + \sum_{v \in T} d_{G-I-S}(v) = ba-a(b+1)+a-1 = -1<0$. By Corollary \ref{co1}, it follows that the graph $F_{n}^{a,b,r}$ is not fractional $ID$-$[a,b]$-factor critical.
\end{proof}
\begin{lemma}(\cite{BA})\label{le:2} \ Let $ G $ be a connected graph and $ H $ be a subgraph of $ G $. Then $ \lambda(H) \leq \lambda(G), $
	with the equality holds if and only if $ H \cong G $.
\end{lemma}

\begin{lemma}(\cite{H})\label{le:3}
	Let $G$ be a connected graph with $n$ vertices. Then
	$\lambda(G) \leq \sqrt{2e(G) - n + 1}.$
\end{lemma}
According to Perron-Frobenius Theorem, for the adjacency matrix $A(G) $ of a connected graph $G$, there exists a positive eigenvector $ \mathbf{x}$ corresponding to $\lambda(G)$. We use $x(v)$ to denote the corresponding entry of the eigenvector $\mathbf{x}$ for every vertex $v \in V(G)$. 
\begin{lemma}(\cite{LLT})\label{le:5} \ Let $G$ be a connected graph and let $u, v$ be two vertices of $G$. 
	Suppose that $v_1, v_2, \dots, v_s \in N_G(v)\setminus N_G(u)$ with $s \geq 1$, and $G^*$ is the graph obtained 
	from $G$ by deleting the edges $vv_i$ and adding the edges $uv_i$ for $1 \leq i \leq s$. Let $\mathbf{x}$ be the 
	Perron vector of $A(G)$. If $x(u) \geq x(v)$, then $\lambda(G) < \lambda(G^*)$.
\end{lemma}

\begin{lemma}\label{le:6}\cite{TFZ}
\	Let $u, v$ be two distinct vertices of a connected graph $G$, and let $\mathbf{x}$ be the Perron vector of $A(G)$.
	\begin{enumerate}
		\item[(i)] If $N_{G}(v)\setminus\{u\} \subset N_{G}(u)\setminus\{v\}$, then $x(u) > x(v)$.
		\item[(ii)] If $N_{G}(v) \subseteq N_{G}[u]$ and $N_{G}(u) \subseteq N_{G}[v]$, then $x(u) = x(v)$.
	\end{enumerate}
\end{lemma}

We present a classical result concerning upper bounds for the spectral radius of a graph.

\begin{lemma}(\cite{HSF,N})\label{le:7} \ Let $G$ be a graph on $n$ vertices and $m$ edges with minimum degree $\delta(G) \geq 1$. Then
	$
	\lambda(G) \leq \frac{\delta(G) - 1}{2} + \sqrt{2e(G) - n\delta(G) + \frac{(\delta(G) + 1)^2}{4}},
	$
	with equality if and only if $G$ is either a $\delta(G)$-regular graph or a bidegreed graph in which each vertex is of degree either $\delta$ or $n - 1$.
\end{lemma}

\begin{lemma}(\cite{HSF,N})\label{le:8} \ For nonnegative integers $p$ and $q$ with $2q \leq p(p - 1)$ and $0 \leq x \leq p - 1$, the function
	$
	f(x) = \frac{x - 1}{2} + \sqrt{2q - px + \frac{(1 + x)^2}{4}}
	$ 
	is decreasing with respect to $x$.
\end{lemma}
 Let $M$ be a real $n \times n$ matrix. Assume that $M$ can be written as the following matrix
\[
M =
\begin{pmatrix}
	M_{1,1} & M_{1,2} & \cdots & M_{1,m} \\
	M_{2,1} & M_{2,2} & \cdots & M_{2,m} \\
	\vdots & \vdots & \ddots & \vdots \\
	M_{m,1} & M_{m,2} & \cdots & M_{m,m}
\end{pmatrix},
\]
whose rows and columns are partitioned into subsets $X_1, X_2, \dots, X_m$ of $\{1, 2, \dots, n\}$. The quotient matrix $R(M)$ of the matrix $M$ (with respect to the given partition) is the $m \times m$ matrix whose entries are the average row sums of the blocks $M_{i,j}$ of $M$. The above partition is called \textit{equitable} if each block $M_{i,j}$ of $M$ has constant row (and column) sum.

\begin{lemma}(\cite{B})\label{le:12}
	Let $M$ be a real symmetric matrix and let $R(M)$ be its equitable quotient matrix. Then the eigenvalues of the quotient matrix $R(M)$ are eigenvalues of $M$. Furthermore, if $M$ is nonnegative and irreducible, then the spectral radius of the quotient matrix $R(M)$ equals the spectral radius of $M$.
\end{lemma}
\section{Proof of Theorem \ref{T1}}

In this section, we first prove Theorem \ref{T1}, which implies that Conjecture \ref{c1} holds when $G$ is connected and $b\ge 2r+2$.

\begin{proof}[\bf Proof of Theorem~\ref{T1}]  %
	
Suppose that $G$ achieves the maximal spectral radius among all connected graphs that are not fractional $ID$-$[a,b]$-factor critical, where $b\ge a\ge1$, $b\ge 2r+2$ and $r\ge1$. By Lemma \ref{le:1}, $H_{n}^{a,r}$ is not fractional $ID$-$[a,b]$-factor critical.  It follows that we aim to prove $G\cong H_{n}^{a,r}$. By Corollary \ref{co1}, there exist disjoint subsets $S, T$ and independent set $I$ of $V (G)$ with $|S| = s$, $|T|=t$ and $|I|=r$ such that
\begin{equation}\label{eq:1}
 \sum_{v \in T} d_{G-I-S}(v) \leq at - bs  - 1.
\end{equation}
 
 Since $K_{n-r-1}$ is a proper subgraph of $H_{n}^{a,r}$ and the maximality of $\lambda(G)$, by Lemma \ref{le:2} we obtain
 \begin{equation}\label{eq:2}
 	\lambda(G) \geq \lambda(H_{n}^{a,r}) > \lambda(K_{n-r-1}) = n - r - 2.
 \end{equation}
	
	By Lemma \ref{le:2} and the maximality of $\lambda(G)$, it is easy to deduce that $d_G(v) = n-r$ for $v\in I$.
	We first divide the proof into two cases according to the value of $d_G(v)$ for $v\in G$.
	
	{\bf Case 1.} There exists a vertex $u\in V(G)$ with $d_{G}(u)\le a+r-1$.
	
Since $d_G(v) = n-r$ for $v\in I$, we obtain $d_{G-I}(u)\le a-1$. Hence, we can deduce that $G-I$ is a spanning subgraph of $K_{a-1}\vee (K_{n-a-r}\cup K_1)$, which implies $G$ is a spanning subgraph of $H_{n}^{a,r}$. By Lemma \ref{le:2}, we obtain $\lambda(G)\le \lambda(H_{n}^{a,r}).$ By (\ref{eq:2}), $\lambda(G)\ge \lambda(H_{n}^{a,r})$, hence $G\cong H_{n}^{a,r}$.
	
	{\bf Case 2.}  $d_G(v)\ge a+r$ for any $v\in V(G)$.
	
By Lemma \ref{le:1}, we have $ 1 \leq s \leq t  - 1$ and $t \geq b + 1.$

	{\bf Claim 1.} $e(\overline{G}) \leq (r+1)(n-\frac{1}{2}(r+3))$. 

	By Lemma \ref{le:3} and (\ref{eq:2}), we obtain
		$n-r-2 \le \lambda(G) \le \sqrt{2e(G) - n + 1},$
	which implies $e(G)>\frac{1}{2}n^2-(r+\frac{3}{2})n+\frac{1}{2}(r+1)(r+3).$
	Hence $e(\overline{G})\le  (r+1)(n-\frac{1}{2}(r+3))$.$\hfill$ $\quad\Box$
	
	Let $W = V(G)\setminus(I\cup S \cup T).$ By Lemma \ref{le:4}, we have
		$	e(G)\leq at-bs-1 + st + \frac{(n -r- t)(n -r- t - 1)}{2}+r(n-r).$
	We now prove the following claim. 
	
	{\bf Claim 2.} $t\le 2(a+r+1)$.
	
	Otherwise, $t\ge 2a+2r+3$. By Lemma \ref{le:3}, $n\ge s+t+r$, $b\ge a\ge1$, $b\ge 2r+2$ we obtain
	\begin{align*}
		 \lambda(G)& \le \sqrt{2e(G) - n + 1} \\
		&\le \sqrt{2(at-bs-1 + st + \frac{(n -r- t)(n -r- t - 1)}{2}+r(n-r)) - n + 1}\\
		&=\sqrt{(n-r-2)^2-f(n)},
	\end{align*}
	where $f(n)=2n(t-r-1)-(t+2s+2a+1)t-(2t-3)r+2bs+2r^2+5.$ We now prove that $f(n)>0$.
	\begin{align*}
		f(n)& =2n(t-r-1)-(t+2s+2a+1)t-(2t-3)r+2bs+2r^2+5 \\
		&\ge t^2-(2a+2r+3)t+2(b-r-1)s+r+5 \quad \text{(since $n\ge s+t+r$)}\\
		&\ge t^2-(2a+2r+3)t+r+5>0 \quad \text{(since $t\ge 2a+2r+3$ and $b\ge 2r+2$)}.
	\end{align*}
	Hence $\lambda(G) \le \sqrt{(n-r-2)^2-f(n)}<n-r-2$, which contradicts (\ref{eq:2}).$\hfill$ $\quad\Box$
	
	Let $e_1$ denote the non-edges inside $E_G(T, W) \cup E_G(T)$. Clearly, $e_1\le e(\overline{G})$. By Claim 1, we obtain 
	$\sum_{v \in T} d_{G-I-S}(v) \geq (n-r-s-1)t - 2e_1 \geq (n-r-s-1)t -(r+1)(2n-(r+3)).$
	
	By Lemma \ref{le:4} and Claim 2, we have $s\le t-1$, $b+1\le t\le 2(a+r+1)$.
	Recall that $n \geq s + t+r$, $b\ge a\ge 1$, $b\ge 2r+2$ and $n \geq (r+1)(b-r)+2a^2+2a(r+2)-1 $, we have
	\begin{align*}
		&\quad bs - at + \sum_{v \in T} d_{G-I-S}(v) \\
		&\geq bs - at +  (n-r-s-1)t -(r+1)(2n-(r+3))\\
		&\ge bs - at +  (n-r-s-1)(b+1) -(r+1)(2n-(r+3))\\
		&=(b-2r-1)n-(r+1)b+(r+3)r-s-at+2 \\
		&\ge n-(r+1)b + (r+3)r-(a+1)t+3 \quad \text{(since $b\ge 2r+2$ and $s\le t-1$)}\\
		&\geq n - (r+1)b + (r+3)r-2(a+1)(a+r+1)+3 \geq 0,
	\end{align*}
	which contradicts (\ref{eq:1}).
	
	This completes the proof of Theorem \ref{T1}.
\end{proof}

\section{Proof of Theorem \ref{T2}}
Before proving Theorem \ref{T2}, we first prove some lemmas. Let $ \mathscr{F}_{n}^{a,b,r}$ be the family of graphs $G$  obtained from
$rK_1\vee (K_{a} \vee ( K_{n-(a+b+r+1)} \cup (b+1)K_1 ))$ by adding $a-1$ edges between $V((b+1)K_1)$ and $a-1$ vertices in $V(K_{n-(a+b+r+1)})$, where $ \mathscr{F}_{n}^{a,b,r}$ is shown in Figure 2.

\begin{lemma}\label{le:2.1}
	Let $a,b,r,n$ be positive integers with $b\ge \max\{r,a\}\ge1$, $n\ge 2(b+r+2)^2$ and let $G\in  \mathscr{F}_{n}^{a,b,r}$. Then
	$\lambda(G) > n - b -\frac{5}{2}.$
\end{lemma}
\begin{proof}
	It is easy to see that $rK_1\vee (K_{a} \vee ( K_{n-(a+b+r+1)} \cup (b+1)K_1 ))$ is a proper subgraph of $G$. By Lemma \ref{le:2}, we have $\lambda(G) > \lambda(rK_1\vee (K_{a} \vee \left( K_{n-(a+b+r+1)} \cup (b+1)K_1 \right)))$.
	
	Note that the partition $ V(K_{a})\cup V(rK_{1}) \cup V(K_{n-(a+b+r+1)}) \cup V((b+1)K_1)$ is an equitable partition of $ V(rK_1\vee (K_{a} \vee \left( K_{n-(a+b+r+1)} \cup (b+1)K_1 \right)))$. Then the quotient matrix of $G$ corresponding to the partition is
	\[
	M_1  =
	\begin{pmatrix}
		a-1 & r & n-a-b-r-1& b+1 \\
		a & 0 & n-a-b-r-1& b+1 \\
		a & r & n-a-b-r-2& 0 \\
		a & r & 0 & 0
	\end{pmatrix}.
	\]
	The characteristic polynomial of $M_1$ is 
	\begin{align*}
		f(x) &= x^{4} - (n-b-r-3) x^{3} - \bigl((r+1)n-r^2-r+a(b+1)-b-2\bigr) x^{2}\\
		&-\bigl(r( b(b-n+r+3)+2)+a^{2}(b+1) +a(b+1)(2+b-n+3r)\bigr) x\\
		&- (a+1)(b+1)(a+b+r+2-n)r
	\end{align*}
	Let $\lambda(M_1)$ be the largest root of the equation $f(x)=0$. We prove that $f(n - b -\frac{5}{2})<0$.
	\begin{align*}
		f(n - b -\frac{5}{2}) =-\frac{n^3}{2}+  \frac{n^2}{4}( 4 r^2+6 b - 6 r +13 ) + \frac{1}{8}nf_1(r)+\frac{1}{16}f_2(r),
	\end{align*}
	where $$f_1(r)= - 8 a^2 (b+1) - 4 b (3 b+13) + 52 r + 28 b r - 
	8 (3 b+5) r^2 - 4 a (b+1) (4r-1)-55,$$
	\begin{align*}
		f_2(r)& =8 a^2 (b+1) ( 2 b - 2 r+5) + ( 2 b+3) (( 2 b+5)^2 - 
		2 (8 b+17) r + 4 ( 4 b+7) r^2)\\
		& + 4 a ( b+1) ( 2 (9 - 2 r) r + b (8 r-2)-5).
	\end{align*}
	We first prove that $f_1(r)<0$. Based on direct computation, $f'_1(r)=52 + 28 b - 16 a (b+1) - 16 (3 b+5) r<0$ due to $b\ge \max\{r,a\}\ge1$. Hence $f_1(r)\le f_1(1)=-43 - 12 a (b+1) - 8 a^2 (b+1) - 12 b (b+4)<0$.
	
	Next we prove that $f_2(r)\le f(b)$	for $r\le b$. Based on direct computation, 
	\begin{align*}
		f'_2(r)& =8r( (2 b+3) (4 b+7)-4 a (b+1))
		+8 a (b+1) ( 4 b+9 - 2 a ) - 
		2 (2 b+3) (8 b+17) \\
		&\ge  2 (2 b+3) ( 8 b+11)+8 a (b+1) ( 4 b+5 - 2 a )>0 \quad (\text{since } r\ge 1).
	\end{align*}
	Hence $f_2(r)\le f_2(b)=32 b^4+16 (a+5) b^3+ 20(4a+1) b^2+ 4(10 a^2+11a+2) b +5 (8 a^2+15 - 4 a )   $
	Hence 	\begin{align*}
		f(n - b -\frac{5}{2}) <-\frac{n^3}{2}+  \frac{n^2}{4}( 4 r^2+6 b  +13 ) +f_2(b) =: g(n).
	\end{align*}
	It is easy to verify that $g(n)\le g(2(b+r+2)^2)$ for $n\ge 2(b+r+2)^2$, then
{\small	\begin{align*}
		g(2(b+r+2)^2)&=-8(b + 2) r^{5}-(6b(6b+23)+131)r^4 - 4(b + 2)(2b + 3)(8b + 17) r^{3} \\
		&- 2(b + 2)^{2}(28b^{2} + 94b + 73) r^{2}
		- 4(b + 2)^{3}\bigl(6b(b + 3) + 11\bigr) r  \\
		& + 40a^{2}(b + 1) + 4a(b + 1)\bigl(4b(b + 4) - 5\bigr)+27\\
		& - b\Bigl(248 + b\bigl(436 + b(312 + b(147 + 42b + 4b^{2}))\bigr)\Bigr)\\
		&\le
		- 4(b + 2)^{3}\bigl(6b(b + 3) + 11\bigr) r + 40a^{2}(b + 1) + 4a(b + 1)\bigl(4b(b + 4) - 5\bigr)+27 .
	\end{align*}}
	Let $y(a)=- 4(b + 2)^{3}\bigl(6b(b + 3) + 11\bigr) r + 40a^{2}(b + 1) + 4a(b + 1)\bigl(4b(b + 4) - 5\bigr)+27 $. Then $y'(a)= 16 b^3+ 80 b^2 + (80 a+44) b +20 ( 4 a-1)>0$. Since $a\le b$, we obtain
	$y(a)\le y(b)=    - 24 b^5- 200 b^4- 644 b^3 - 1236 b^2- 1124 b-325<0.$
	Hence $$	f(n - b -\frac{5}{2})< g(2(b+r+2)^2)<y(a)\le y(b)<0,$$
	which implies that $\lambda(M_1)>n - b -\frac{5}{2}$. By Lemmas  \ref{le:2} and  \ref{le:12}, we have 
	
	$\lambda(G) > \lambda(rK_1\vee (K_{a} \vee \left( K_{n-(a+b+r+1)} \cup (b+1)K_1 \right)))=\lambda(M_1)>n - b -\frac{5}{2}.$
\end{proof}
\begin{lemma}\label{le:2.2}
	Let $r,b,a,n$ be positive integers with $r \geq 1$, $b\ge a\ge r+2$ and $n \geq a+(a+r+2)b$. If $G \in \mathscr{F}_{n}^{a,b,r}$, then
	$\lambda(G) \leq \lambda(F_{n}^{a,b,r}),$
	with equality if and only if $G \cong F_{n}^{a,b,r}$.
\end{lemma}
\begin{proof}
	Suppose that \( G \)  is the extremal graph with the largest spectral radius in $ \mathscr{F}_{n}^{a,b,r}$. Clearly, $F_{n}^{a,b,r} \in \mathscr{F}_{n}^{a,b,r}$. Our aim is to verify that  \( G \cong F_{n}^{a,b,r} \). Since $G\in  \mathscr{F}_{n}^{a,b,r}$, by direct calculation, we obtain
	
	\begin{equation}\label{eql1}
		e(G)= \binom{n-r-b-1}{2} +a(b+1)+a-1+r(n-r).
	\end{equation}
	Since \( K_{n-r-b-1} \) is a proper subgraph of \( G \), by Lemma \ref{le:2}, we obtain 
	\begin{equation}\label{eql2}
		\lambda(G) > \lambda(K_{n-r-b-1}) = n-r-b-2.
	\end{equation}
	
	Now we prove the following claim:
	
	{\bf Claim 1.} $\lambda(G)<n-b-1$.
	
	Since $n \geq a+(a+r+2)b$, $b\ge a\ge r+2$ and $\delta(G)\ge a+r$, by (\ref{eql1}), Lemmas \ref{le:7} and \ref{le:8}, we obtain
	{\footnotesize \begin{align*}
			\lambda(G)& \leq \frac{\delta(G) - 1}{2} + \sqrt{2e(G) - n\delta(G) + \frac{(\delta(G) + 1)^2}{4}}\\
			&\leq \frac{a+r-1}{2} + \sqrt{2\Big(\binom{n-r-b-1}{2} +a(b+1)+a-1+r(n-r)\Big)- n(a+r) + \frac{(a+r+1)^2}{4}} \\
			&= \frac{a+r-1}{2} + \sqrt{\left(n - b - \frac{a+r+1}{2}\right)^2 - (2n-(r+2)b-(b+4)a+r(r-3))} \\
			&< \frac{a+r-1}{2} + \left(n - b - \frac{a+r+1}{2}\right) \quad (\text{since } n \geq a+(a+r+2)b ) \\
			&= n - b - 1.
	\end{align*}}
	Hence $\lambda(G)<n-b-1$. $\hfill$ $\quad\Box$
	
	Let $V(G)=S\cup I\cup T\cup W,$ where $S=V(K_{a})=\{s_{1},s_{2},\ldots,s_{a}\}$, $I=V(rK_1)=\{u_1,u_2,\cdots,u_r\}$, $W=V(K_{n-a-r-b-1})=\{w_{1},w_{2},\ldots,w_{n-a-r-b-1}\}$ and 
	$T=\{V((b+1)K_1)\}=\{t_{1},t_{2},\ldots,t_{b+1}\}$. 
	We use $\mathbf{x}$ to denote the Perron vector of $A(G)$. 
	Without loss of generality, suppose that $x(w_{i})\geq x(w_{i+1})$ and 
	$x(t_{j})\geq x(t_{j+1})$ for $1\leq i\leq n-a-r-b-2$ and $1\leq j\leq b$. 
	Then we have $N_{G}(w_{i+1})\subseteq N_{G}[w_{i}]$ for $a+1\leq i\leq n-r-a-b-2$. Otherwise, there exist $i<j$ such that	$N_{G}(w_{j})\nsubseteq N_{G}[w_{i}]$. Let $v \in N_{G}(w_{j})\setminus N_{G}[w_{i}]$ and let $G^{*} = G - vw_{j} + vw_{i}$. Clearly, $G^{*} \in \mathscr{F}_{n}^{a,b,r}$. Since $x(w_{i}) \geq x(w_{j})$, by Lemma \ref{le:5}, we obtain 	$\lambda(G^{*}) > \lambda(G)$, which contradicts the maximality of $\lambda(G)$. Hence $N_{G}(w_{i+1}) \subseteq N_{G}[w_{i}]$ for $1 \leq i \leq n-a-r-b-2$. 	Similarly, we have $N_{G}(t_{j+1}) \subseteq N_{G}[t_{j}]$ for $1 \leq j \leq b$.
	
	Let $d_{W}(t_{1}) =l$. By the maximality of $\lambda(G)$ and Lemma \ref{le:5}, we obtain $N_{W}(t_{1}) = \{w_{1}, w_{2}, \ldots, w_{l}\}$.  Otherwise, there exists a vertex $w_p\in N_{W}(t_1)$ and  a vertex $w_k\notin N_{W}(t_1)$, where $p\ge l+1$ and $1\le k\le l$. Recall that $x(w_{1}) \geq x(w_{2}) \geq \cdots \geq x(w_{n-a-r-b-1})$, we set $G_1=G-t_1w_p+t_1w_k$. Then $G_1 \in \mathscr{F}_{n}^{a,b,r}$ and $\lambda(G_1)>\lambda(G)$ due to Lemma \ref*{le:5}, which contradicts the maximality of $\lambda(G)$. Since $N_{G}(t_{j+1}) \subseteq N_{G}[t_{j}]$ for $1 \leq j \leq b$ and $e(T)=0$, we obtain $N_{W}(t_{j})\subseteq N_{W}(t_{1}) = \{w_{1}, w_{2}, \ldots, w_{l}\}$ for $2 \leq j \leq b$.
	
	{\bf Case 1.} $l=a-1$.
	
	In this case, note that there are $a-1$ edges between $V((b+1)K_1)$ and $V(K_{n-a-r-b-1})$. If \( l = a-1 \), then it means that these \( a-1 \) edges are exactly the edges connecting \( t_1 \) with \( a-1 \) vertices in \(W \), which implies $G \cong F_{n}^{a,b,r}$, as required.
	
	{\bf Case 2.}  $l\leq a-2$. 
	
	By symmetry, we obtain $x(u_i)=x(u_1)$ for $2\le i\le r$, $x(s_i)=x(s_1)$ for $2\le i\le a$ and $x(w_{i})=x(w_{l+1})$ for $l+2\le i\le n-a-r-b-1$.
	For $s_1\in S$ and $u_1\in I$,	by $A(G)\mathbf{x} = \lambda(G) \mathbf{x}$, we obtain
	$$\lambda(G)x(s_1)=rx(u_1)+(a-1)x(s_1)+\sum_{v\in T}x(v)+\sum_{v\in W}x(v),$$
	$$\lambda(G)x(u_1)=ax(s_1)+\sum_{v\in T}x(v)+\sum_{v\in W}x(v).$$
	Hence 
	\begin{equation*}
		x(s_1)=\Big(\frac{r-1}{\lambda(G)+1}+1\Big)x(u_1)
	\end{equation*}
	By (\ref{eql2}), we obtain $\lambda(G) > n-r-b-2.$ Recall that $b\ge a\ge r+2$ and $n \geq  a+(a+r+2)b$, we easily obtain $\lambda(G)+1 >n-r-b-1> 9r$. Hence
	\begin{equation}\label{eql3}
		x(u_1)\le \Big(\frac{r-1}{\lambda(G)+1}+1\Big)x(u_1)=x(s_1)<\Big(\frac{r}{\lambda(G)+1}+1\Big)x(u_1)<\frac{10}{9}x(u_1).
	\end{equation}
	For $w_{i}\in W$ and $s_1\in S$, since $N_{G}(w_{i}) \subseteq N_{G}[s_1]$, by Lemma \ref{le:6}, we have $x(s_1)>x(w_{i})$. For $w_{l+1}\in W$, by (\ref{eql3}) and $A(G)\mathbf{x} = \lambda(G) \mathbf{x}$, we obtain	
	\begin{equation*}
		\begin{aligned}
			\lambda x(w_{l+1})&=ax(s_1)+rx(u_{1})+x(w_{1})+\sum_{2\leq i\leq l}x(w_{i})+(n-a-r-b-2-l)x(w_{l+1})\\
			&\ge (a+\frac{9}{10}r+1)x(w_{1})+(n-a-r-b-3)x(w_{l+1}).
		\end{aligned}
	\end{equation*}
	By Claim 1, we have $\lambda(G)<n-b-1$. Hence
	\begin{equation}\label{eql4}
		x(w_{l+1})\ge \frac{a+\frac{9}{10}r+1}{\lambda+ a+r+b+3-n}x(w_{1})>\frac{a+\frac{9}{10}r+1}{a+r+2}x(w_{1}).
	\end{equation}
	We construct $F_{n}^{a,b,r}$ by deleting the edges $t_iw_j$($i\ge2$) and connecting these vertices $w_j$ to $t_1$. Recall that $N_{W}(t_{j})\subseteq  \{w_{1}, w_{2}, \ldots, w_{l}\}$ for $1 \leq j \leq b$. Suppose that $E_{1} = \{t_{1}w_{i} \mid l+1 \leq i \leq a-1\}$ and 
	$E_{2} = \{t_{i}w_{j} \in E(G) \mid  2\le i\le b+1,  2\le j \le l\}$. 
	Let $G^{*} = G - E_{2} + E_{1}$. Then $G^{*} \cong F_{n}^{a,b,r}$. 
	Let $\mathbf{y}$ be the Perron vector of $A(G^{*})$.
	Note that
	$y(s_{i}) = y(s_{1})$ for $2 \leq i \leq a$,
	$y(u_{i}) = y(u_1)$ for $2 \leq i \leq r$,
	$y(t_{i}) = y(t_{2})$ for $3 \leq i \leq b+1$,
	$y(w_{i}) = y(w_{1})$ for $2 \leq i \leq a-1$ and
	$y(w_{i}) = y(w_{a})$ for $a+1 \leq i \leq n-a-r-b-1$.
	Similarly to the proof of  (\ref{eql3}), we have
	\begin{equation}\label{eql5}
		y(u_1)\le \Big(\frac{r-1}{\lambda(G^{*})+1}+1\Big)y(u_1)=y(s_1)<\Big(\frac{r}{\lambda(G^{*})+1}+1\Big)y(u_1)<\frac{10}{9}y(u_1).
	\end{equation}
	By $A(G^{*})\mathbf{y} = \lambda(G^{*})\mathbf{y}$ and (\ref{eql4}), we obtain
	\begin{equation*}
		\begin{aligned}
			\lambda(G^{*}) y(w_{a})&=ay(s_1)+ry(u_{1})+(a-1)y(w_{1})+(n-2a-r-b-1)y(w_{a})\\
			&\ge (a+\frac{9}{10}r)y(s_{1})+(n-a-r-b-2)y(w_{a}).
		\end{aligned}
	\end{equation*}
	Hence 
	\begin{equation}\label{eql6}
		y(w_{1})\ge	y(w_{a})\ge \frac{a+\frac{9}{10}r}{\lambda(G^{*})+b+a+r+2-n}y(s_{1}).
	\end{equation}
	Since $\lambda(G^{*}) y(t_{1})=ay(s_1)+ry(u_{1})+(a-1)y(w_{1})$, we have
	$$\lambda(G^{*}) y(t_{1})\ge (a+\frac{9}{10}r)y(s_{1})+(a-1)\frac{a+\frac{9}{10}r}{\lambda(G^{*})+b+a+r+2-n}y(s_{1}).$$
	
	Note that $\lambda(G^{*}) y(t_{2})=ay(s_1)+ry(u_{1})\le (a+r)y(s_1)$ due to (\ref{eql5}) and $\frac{a+\frac{9}{10}r}{a+r}=1-\frac{1}{10}\frac{r}{a+r}>1-\frac{1}{10}=\frac{9}{10}$, we have
	
	\begin{equation}\label{eql7}
		\begin{aligned}
			\frac{y(t_{1})}{ y(t_{2})}=	\frac{\lambda(G^{*}) y(t_{1})}{\lambda(G^{*}) y(t_{2})}&\ge \frac{a+\frac{9}{10}r}{a+r}+\frac{(a-1)(a+\frac{9}{10}r)}{(a+r)(\lambda(G^{*})+b+a+r+2-n)}\\
			&=\frac{a+\frac{9}{10}r}{a+r}\cdot\frac{\lambda(G^{*})+b+2a+r+1-n}{\lambda(G^{*})+b+a+r+2-n}\\
			&>\frac{9}{10}\cdot\frac{\lambda(G^{*})+b+2a+r+1-n}{\lambda(G^{*})+b+a+r+2-n}
		\end{aligned}
	\end{equation}
	Since $x(t_{1})\geq x(t_{i})$ for $2\leq i\leq b+1$, $x(w_{1})\geq x(w_{j})$ for $2\leq j\leq n-a-r-b-1$, by (\ref{eql6}) and (\ref{eql7}), we have
	{\small	\begin{align*}
			& \mathbf{y}^{T}(\lambda(G^{*})-\lambda(G))\mathbf{x} \\
			&= \mathbf{y}^{T}(A(G^{*})-A(G))\mathbf{x}\\
			&= \sum_{t_{1}w_{i}\in E_{1}}(x(t_{1})y(w_{i})+x(w_{i})y(t_{1}))-\sum_{t_{i}w_{j}\in E_{2}}(x(t_{i})y(w_{j})+x(w_{j})y(t_{i})) \\
			&\geq (a-1-l)(x(t_{1})y(w_{1})+x(w_{l+1})y(t_{1})-(a-1-l)(x(t_{2})y(w_{1})+x(w_{1})y(t_{2})) \\
			&\geq (a-1-l)(x(w_{l+1})y(t_{1})-x(w_{1})y(t_{2})) \quad (\text{since } x(t_{1})\geq x(t_{2})) \\
			&>
			(a-1-l)x(w_{1})y(t_{2})\frac{9}{10} \bigg(\frac{a+\frac{9}{10}r+1}{a+r+2}\cdot\frac{\lambda(G^{*})+b+2a+r+1-n}{\lambda(G^{*})+b+a+r+2-n}-\frac{10}{9}\bigg)  (\text{by (\ref{eql6}) and (\ref{eql7})}) \\
			&= (a-1-l)x(w_{1})y(t_{2})\frac{9}{10} \bigg(\frac{f(n)}{9(a+r+2)(\lambda(G^{*})+b+a+r+2-n)}\bigg),
	\end{align*}}
	where $f(n)= \Big(n-\lambda(G^{*})\Big)\left(a + \frac{19r}{10} + 11\right)
	+\frac{- 19r^2+ \left( 52a - 19b - 229 \right) r+10 \left( 8a^2 - 11b - a(b + 13) - 31 \right)}{10}  .$ 
	
	Since $\lambda(G^{*})<n-b-1$, then
	\begin{equation*}
		\begin{aligned}
			f(n)&\ge 8a^2 + a\left(\frac{26r}{5} - 12\right) - \frac{1}{10}(r + 10)(19r + 20)\quad \text{(since $n-\lambda(G^{*})>b+1$)}\\
			&\ge \frac{r}{10} (113 r+94)-12>0\quad \text{(since $a\ge r+2$ and $r\ge1$)}.
		\end{aligned}
	\end{equation*}
	Hence $f(n)>0$, which implies that $\lambda(G^{*})>\lambda(G).$ This contradicts the maximality of $\lambda(G)$. This completes the proof.
\end{proof}

Now, we give the proof of Theorem \ref{T2}.

\begin{proof}[\bf Proof of Theorem~\ref{T2}]  %
Suppose that $G$ achieves the maximal spectral radius among all connected graphs that are not fractional $ID$-$[a,b]$-factor critical, where $b\ge a\ge r+2$ and $r\ge1$. By Corollary \ref{co1}, there exist disjoint subsets $S, T$ and independent set $I$ of $V (G)$ with $|S| = s$, $|T|=t$ and $|I|=r$ such that
\begin{equation}\label{eq:2.1}
	\sum_{v \in T} d_{G-I-S}(v) \leq at - bs  - 1.
\end{equation}
Subject to (\ref{eq:2.1}), we choose $S$ such that $|S|=s$ is as large as possible. By Lemma \ref{le:1}, we have 
\begin{equation}\label{eq:2.2}
	1 \leq s \leq t  - 1, t \geq b + 1.
\end{equation}

Recall that $F_{n}^{a,b,r}$ is not fractional $ID$-$[a,b]$-factor critical. According to the maximality of $\lambda(G)$ and Lemma \ref{le:2.1}, we have
\begin{equation}\label{eq:2.3}
	\lambda(G) \geq \lambda(F_{n}^{a,b,r}) > n-b-\frac{5}{2}.
\end{equation}

Now we prove some claims.

{\bf Claim 1.} $e(G)>\frac{1}{2}n^2-(b+2)n+\frac{1}{8}(2 b+3)  (2 a + 2 b + 2 r+5)$.

By integrating $\delta(G) \geq a+r$, (\ref{eq:2.3}), Lemmas \ref{le:7} and \ref{le:8},  we obtain
$	n-b-\frac{5}{2}<\lambda(F_{n}^{a,b,r}) \le \lambda(G) \le  \frac{a+r-1}{2} + \sqrt{2e(G) - n(a+r) + \frac{(a+r+1)^2}{4}},$
which implies $e(G)>\frac{1}{2}n^2-(b+2)n+\frac{1}{8}(2 b+3)  (2 a + 2 b + 2 r+5)$.$\hfill$ $\quad\Box$

 Again by the maximality of $\lambda(G)$ and Lemma \ref{le:2}, we can deduce that $G[V(G)\setminus (I\cup T)] \cong K_{n-r-t}$, $G[S,T] \cong K_{s,t}$ and $d_G(v)=n-r$ for $v\in I$. Let $W = V(G)\setminus(I\cup S \cup T).$ By Lemma \ref{le:4}, we have
\begin{equation}\label{eq:2.4}
	e(G)\leq at-bs-1 + st + \frac{(n -r- t)(n -r- t - 1)}{2}+r(n-r).
\end{equation}

{\bf Claim 2.} $t=b+1$.

Note that $t\ge b+1$ and $b\ge a\ge r+2$. Assume $t\ge b+2$. We consider the following cases according to the value of $t$.

{\bf Case 1.} $b+2\le t\le a+b-1.$

By Cliam 1 and (\ref{eq:2.4}), we obtain
{\footnotesize\begin{align*}
	0 &> \frac{1}{2}n^2-(b+2)n+\frac{1}{8}(2 b+3)  (2 a + 2 b + 2 r+5) - e(G) \\
	&\ge (t - b- \frac{3}{2})n -\frac{t^2}{2}-(a+s+r+\frac{1}{2})t+\frac{4 (b^2 + r^2 + b ( r + 2 s+4))+ 2 r+ 2a ( 2b+3) +23}{8}
	 \quad \text{(by (\ref{eq:2.4}))} \\
	&\ge (t - b- \frac{3}{2})n
	-\frac{3}{2} t^{2} + (  b -a - r + \frac{1}{2}) t + \frac{4\bigl( b^{2} + r^{2} + b(r + 2) \bigr)+2(2b+3)a + 2r +23}{8} 
	\quad \text{(since $s\le t-1$)} \\
	&\ge \frac{n}{2} - \frac{20 a^2 + 4a (2r+5b-13) + 4 b (r-7) - 2 r ( 2 r+5)-7}{8} \quad \text{(since $b+2\le t\le a+b-1$)} \\
	&>0 \quad \text{(since $n\ge 2(4b+a+r+2)(b+r+2)$),}
\end{align*}}
a contradiction.

{\bf Case 2.} $a+b\le t<2(4b+a+r+2).$

Since $n\ge 2(4b+a+r+2)(b+r+2)$, we obtain $n\ge(b+r+2)t+1$.
{\footnotesize\begin{align*}
	0 &> \frac{1}{2}n^2-(b+2)n+\frac{1}{8}(2 b+3)  (2 a + 2 b + 2 r+5) - e(G) \\
	&\ge (t - b- \frac{3}{2})n -\frac{t^2}{2}-(a+s+r+\frac{1}{2})t\\
	&+\frac{4 (b^2 + r^2 + b ( r + 2 s+4))+ 2 r+ 2a ( 2b+3) +23}{8}
	\quad \text{(by (\ref{eq:2.4}))} \\
	&\ge ( b + r+\frac{3}{2} ) t^{2} - \left( a+ \frac{1}{2}(2b + 5)(b + r + 1) + s \right) t\\
	& + \frac{4\Bigl( b^{2} + r^{2} + b(r + 2s + 2) \Bigr) + 2(2b + 3) + 2r + 11 }{8}  \quad \text{(since $n\ge(b+r+2)t+1$)}\\
	&\ge ( b + r+\frac{1}{2} ) t^{2} - \left(  \frac{1}{2}(2b + 5)(b + r + 1)+a-b-1 \right) t \\
	&+ \frac{ 4( b^{2} + br + r^{2} )+ 2a(2b + 3) + 2r+11}{8}  \quad \text{(since $s\le t-1$)}\\
	&>\left( (b + r+\frac{1}{2} ) t -   \frac{1}{2}(2b + 5)(b + r + 1) +1\right) t \quad \text{(since $b\ge a$)}.
\end{align*}}
Let $f(t)=(b + r+\frac{1}{2} ) t -   \frac{1}{2}(2b + 5)(b + r + 1) +1$. Since $t\ge a+b$, we have
$f(t)\ge f(a+b)=(a-3) b + \frac{1}{2}(a + 2 a r - 5 (1 + r))+1$. Now we prove that $f(a+b)>0$. If $a\ge 4$, then obviously $f(a+b)>0.$ If $a=3$, since $a\ge r+2$, we obtain $r=1$. Hence $f(3+b)=-\frac{1}{2}+1>0.$ Hence
$f(t)>f(a+b)>0$, which implies
$0 > \frac{1}{2}n^2-(b+2)n+\frac{1}{8}(2 b+3)  (2 a + 2 b + 2 r+5) - e(G)>f(t)t>0,$
a contradiction.

{\bf Case 3.} $t\ge2(4b+a+r+2)$

 By integrating (\ref{eq:2.4}), $\delta(G) \geq a+r$, Lemmas \ref{le:7} and \ref{le:8}, we have
{\small\begin{align*}
	\lambda(G) &\leq \frac{\delta(G) - 1}{2} + \sqrt{2e(G) - n\delta(G) + \frac{(\delta(G) + 1)^2}{4}} \\
	&\leq \frac{a+r-1}{2}\\
	& + \sqrt{2\left(at-bs-1 + st + \binom{n -r- t}{2}+r(n-r)\right) - n(a+r) + \frac{(a+r+1)^2}{4}} \\
	&=\frac{a+r-1}{2}+\sqrt{\left(n-b-\frac{a+r+4}{2}\right)^2 - f(n)}\\
\end{align*}}
where $f(n)=(2t-2b-3)n-t^2-(2(s+a+r)+1)t+b^2+b(2s+r+4)+ab+\frac{3}{2}a+ r^2+\frac{r}{2}+\frac{23}{4}.$ Now we prove that $f(n)>0$.
{\small \begin{align*}
	f(n)&=(2t-2b-3)n-t^2-(2(s+a+r)+1)t+b^2+b(2s+r+4)+ab+\frac{3}{2}a+ r^2+\frac{r}{2}+\frac{23}{4} \\
	&\ge t^2 - 2 ( a + b+2) t- 3 s + r^2+b(b-r+4)+ab+\frac{3}{2}a -\frac{5r}{2}+\frac{23}{4} \quad \text{(since $n\ge s+t+r$)}\\
	&\ge t^2 - (2  a + 2b+7) t + r^2+b(b-r+4)+ab+\frac{3}{2}a -\frac{5r}{2}+\frac{35}{4} \quad \text{(since $s\le t-1$)}\\
	&\ge 5 r^2+ ( 27 b+4a-\frac{1}{2}) r+ 
	b (49 b+4)  + a ( 13 b-\frac{9}{2})-\frac{13}{4}>0
  \quad \text{(since $t\ge 2(4b+a+r+2)$)}.
\end{align*}}
Hence $\lambda(G)\le\frac{a+r-1}{2}+\sqrt{\left(n-b-\frac{a+r+4}{2}\right)^2 - f(n)}<n-b-\frac{5}{2}$, which leads to a contradition with (\ref{eq:2.3}).$\hfill$ $\quad\Box$

Since $t=b+1$, by Lemma \ref{le:1}, we have $s\le t-1=b$.  Let $\mathbf{x}$ be the Perron vector of $A(G)$. Let 
$W = V(G)\setminus(S \cup I\cup T) = \{w_1, w_2, \ldots, w_{n-s-r-b-1}\}.$ Without loss of generality, we suppose that $x(w_1) \geq x(w_2) \geq \cdots \geq x(w_{n-s-r-b-1})$.

{\bf Claim 3.} $e(T)=0.$

 Otherwise,  there exists $uv \in E(T)$.  Since $n\ge 2(4b+a+r+2)(b+r+2)$, $t= b+1$ and $s\le t-1=b$, we have $|W|=n-s-r-t >a(b+1)-b-1$. Then we have
$d_W(u) \leq \sum_{v \in T} d_{G-I-S}(v) \leq at-bs-1\le a(b+1)-b-1<|W|.$ Hence, there exists a vertex $w_i \in W$ such that $uw_i \notin E(G)$. Suppose that $c \in T$ with $x(c) = \max\{x(v) \mid v \in T\}$. Let $d_T(c) = l$. Since $c \in T$, we have $d_{G-I-S}(c) \leq a-1$ due to $T = \{v \in V(G) \setminus (S\cup I) \mid d_{G-I-S}(v) \leq a-1\}$.

By $\lambda(G)\mathbf{x} = A(G)\mathbf{x}$, we obtain
\begin{align*}
	\lambda(G)x(c)&=\sum_{v\in S}x(v)+\sum_{v\in I}x(v)+\sum_{v\in N_W(c)}x(v)+\sum_{v\in N_T(c)}x(v)\\
	&\le \sum_{v \in S} x(v)+\sum_{v\in I}x(v)+\sum_{1 \leq i \leq a-1-l} x(w_i) + lx(c),
\end{align*}
\begin{align*}
	\lambda(G)x(w_{n-s-b-1})&=\sum_{v\in S}x(v)+\sum_{v\in I}x(v)+\sum_{v\in N_W(w_{n-s-b-1})}x(v)+\sum_{v\in N_T(w_{n-s-b-1})}x(v)\\
	&\ge \sum_{v \in S} x(v)+\sum_{v\in I}x(v)+\sum_{1 \leq i \leq a-1-l} x(w_i) + (n-s-b-1-a+l)x(w_{n-s-b-1}).
\end{align*}

Since $n\ge 2(4b+a+r+2)(b+r+2)$ and $s\le b$, we have	
$$\left(\lambda(G)-l)(x(w_{n-s-b-1}) - x(c)\right) \geq (n-s-b-1-a)x(w_{n-s-b-1}) > 0.$$

Note that $l=d_T(c)\le d_{G-I-S}(c)\le a-1$, by (\ref{eq:2.3}), we have $\lambda(G) > n-b-\frac{5}{2} > a-1 \geq l$. Hence $x(w_{n-s-b-1}) > x(c)$. Since  $x(w_i) \geq x(w_{n-s-b-1})$ and $x(c) \geq x(v)$, we have $x(w_i) > x(v)$. Let $G' = G-uv+uw_i$. Since
$\sum_{v \in T} d_{G'-I-S}(v) = \sum_{v \in T} d_{G-I-S}(v) - 1 <at - bs-1,$
by Corollary \ref{co1} and Lemma \ref{le:5}, we deduce that $G'$ is not fractional $ID$-$[a,b]$-factor critical and $\lambda(G') > \lambda(G)$, which contradicts the maximality of $\lambda(G)$. Hence $e(T)=0.$ $\hfill$ $\quad\Box$

{\bf Claim 4.} $s =a$.

 If $s \geq a + 1$, for any $v\in T$, since $t=b+1$ and $b\ge a$, we obtain
$$0 \leq \sum_{v \in T} d_{G-I-S}(v) \leq at-bs-1\leq a(b+1)-b(a+1)-1=a-b-1 < 0,$$
 a contradiction.

If $s \leq a- 1$, since $d_G(v) \geq \delta(G) \geq a+r$ and $e(T)=0$ due to Cliam 3, then $d_{G-I-S}(v) = d_W(v) \geq a+r-r - s=a-s$ for $v \in T$. Since $W = \{w_1, w_2, \ldots, w_{n-s-r-b-1}\}$ with $x(w_1) \geq x(w_2) \geq \cdots \geq x(w_{n-s-r-b-1})$, by Lemma \ref{le:5} and the maximality of $\lambda(G)$, we obtain $\{w_1, w_2, \ldots, w_{a-s}\} \subseteq N_G(v)$ for any $v \in T$.

Let $S' = S \cup \{w_1, w_2, \ldots, w_{a-s}\}$ . Then $|S'|=s+a-s=a.$

Since $s \leq a-1$, we have 
$$\sum_{v \in T} d_{G-I-S'}(v) = \sum_{v \in T} d_{G-I-S}(v) - (a- s)(b + 1) \leq s-1 \le a-2,$$
$$at-b|S'|-1=a(b+1)-ba-1=a-1. $$
Hence $|S'|>s$ also satisfies $\sum_{v \in T} d_{G-I-S'}(v)<at-b|S'|-1,$
 which contradicts the maximality of $s$. Thus $s = a$.$\hfill$ $\quad\Box$
 
  Combining this with $t = b + 1$ and $\sum_{v \in T} d_{G-I-S}(v) \le at-bs-1=a-1 ,$ by the maximality of $\lambda(G)$, we have $\sum_{v \in T} d_{G-I-S}(v)=a-1 $. Hence $G\in \mathscr{F}^{a,b,r}_n$. Again by the maximality of $\lambda(G)$ and Lemma \ref{le:2.2}, we get $G \cong F^{a,b,r}_n$, as required. 
  
  This completes the proof of Theorem \ref{T2}.
\end{proof}

\section{Proof of Theorem \ref{T3}}

In this section, we prove Theorem \ref{T3}, which characterizes the size condition for a graph to be fractional $ID$-$[a,b]$-factor critical.

\begin{proof}[\bf Proof of Theorem~\ref{T2}]  %
Suppose that $G$ is not fractional $ID$-$[a,b]$-factor critical, where $b\ge a\ge 1$ and $r\ge1$. By Corollary \ref{co1}, there exist disjoint subsets $S, T$ and independent set $I$ of $V (G)$ with $|S| = s$, $|T|=t$ and $|I|=r$ such that
\begin{equation}\label{eq:3.1}
	\sum_{v \in T} d_{G-I-S}(v) \leq at - bs  - 1.
\end{equation}
 By Lemma \ref{le:1}, we have 
\begin{equation}\label{eq:3.2}
	1 \leq s \leq t  - 1, t \geq b + 1.
\end{equation}

By Lemma \ref{le:4}, we obtain 
	\begin{equation}\label{eq:3.3}
		\begin{aligned}
			e(G)&\le at-bs-1 + st + \frac{(n -r- t)(n -r- t - 1)}{2}+r(n-r)\\
			&= \binom{n-r-b-1}{2}+r(n-r)+a(b+1)+a-y(t),
		\end{aligned}
	\end{equation}
where $y(t) = -\frac{t^2}{2} + \left(n - a-r- s - \frac{1}{2}\right)t + \frac{b^2}{2}+ab+r-(b+1)n + \frac{1}{2}b(2r+2s+3) +2a + 2.$
	
{\bf Claim 1.} $y(t)>0$.
	
	{\bf Case 1.} $b+1\le t\le \frac{n}{2}.$

	{\bf Subcase 1.1.} $t=b+1$.
	
In this case, we obtain $y(t)=y(b+1)=a+1-s$. We prove that $s \le a$. Otherwise, $s \ge a+1$. By (\ref{eq:3.1}) and $b\ge a$, we have
	$$
	0 \leq \sum_{v \in T} d_{G-I-S}(v) \leq at-bs-1\le a(b+1)-b(a+1)-1<a-b-1<0,
	$$
	a contradiction. Hence $s \le a$, which implies $y(t)=y(b+1)=a+1-s>0.$
	
	{\bf Subcase 1.2.} $b+2\le t\le \frac{n}{2}.$
	
By (\ref{eq:3.2}),	we have $s\le t-1$ and $\frac{\partial f}{\partial s}=-t+b<0.$ By direct calculation,
	\begin{align*}
		y(t) &\geq - \frac{3}{2} t^{2} + \left(   b + n  -a-r + \frac{1}{2} \right) t  + a(b + 2) - (b + 1)n + r + \frac{1}{2}b (b+ 2r + 1)+ 2=:q(t).
	\end{align*}
	Since $n \geq 4(a+b+r+1)$ and $b \ge a \geq 1$, we obtain 
	$$
	q(b + 2) = n-3b-r-3 > 0
	$$
	and
	\begin{align*}
		q\left(\frac{n}{2}\right) &= \frac{n^2}{8} -\frac{n}{4} (2a + 2b + 2r + 3) + a(b + 2) + r + \frac{1}{2} b(b + 2r + 1)+ 2 \\
		&\geq   \frac{1}{2} b(b + 2r - 1) + a(b + 1) +1 >0
	\end{align*}
	Hence
	$
	q(t) \geq \min\left\{q(b + 2), q\left(\frac{n}{2}\right)\right\} > 0.$ 
	
	Hence $y(t) \geq q(t) > 0$ for $b + 1 \leq t \leq \frac{n}{2}$.
	
	{\bf Case 2.} $t\ge \frac{n+1}{2}.$
	
	Since $n \geq s + t+r$. Then $s \leq n - t-r$. Since $t \geq \frac{n+1}{2}$, we get
	\begin{align*}
		y(t) &\geq \frac{t^2}{2} -\left( a + b + \frac{1}{2} \right) t + a(b + 2) + \frac{1}{2} b(b + 3) - n + r  +  2 \\
		&\geq \frac{n^2}{8} - \frac{1}{2}(a + b + 2)n  + b + \frac{b^2}{2} + (b + \frac{3}{2})a + r +  \frac{15}{8}  \quad \text{(since $t \geq \frac{n+1}{2}$)} \\
		&\geq   \frac{b^2}{2} + (b + \frac{3}{2}) a + b+ r+\frac{15}{8}>0 \quad \text{(since $n \geq 4(a+b+r+1)\ge 4(a+b+2)$)} .
	\end{align*}  
	Hence $y(t)>0$ for $t\ge b+1$.  $\quad\Box$
	
According to (\ref{eq:3.3}), we have $e(G) < \binom{n-r-b-1}{2}+r(n-r)+a(b+1)+a$, which leads to a contradiction. 

This completes the proof of Theorem \ref{T3}.
\end{proof}

\section*{Declarations}

The authors declare that they have no conflict of interest.

\section*{Data availability}

No data was used for the research described in the paper.












\begin{thebibliography}{99}

\bibitem{AK} J. Akiyama, M. Kano, Factors and Factorizations of Graphs: Proof Techniques in Factor Theory, vol. 2031, Springer, 2011.

\bibitem{A} R.P. Anstee, Simplified existence theorems for $(g, f )$-factor, Discrete Appl. Math. 27 (1990) 29–38.

\bibitem{BA} R.B. Bapat, Graphs and Matrices, Springer, London, 2014.

\bibitem{BM} J.A. Bondy, U.S.R. Murty, Graph Theory with Applications, Macmillan, London, 1976.

\bibitem{BH} A. Brouwer, W. Haemers, Eigenvalues and perfect matchings, Linear Algebra Appl. 395 (2005) 155–162.

\bibitem{B} A. Brouwer, W. Haemers, Spectra of Graphs, Springer, Berlin, 2011.

\bibitem{C} S. Cioab\v{a}, Perfect matchings, eigenvalues and expansion, C. R. Math. Acad. Sci. Soc. R. Can. 27 (4) (2005) 101–104.

\bibitem{CG} S. Cioab\v{a} D. Gregory, Large matchings from eigenvalues, Linear Algebra Appl. 422 (1) (2007) 308–317.

\bibitem{CGH} S. Cioab\v{a}, D. Gregory, W. Haemers, Matchings in regular graphs from eigenvalues, J. Combin. Theory, Ser. B 99 (2) (2009) 287–297.

\bibitem{FLA}	A. Fan, R. Liu, G. Ao, Spectral radius, fractional $[a, b]$-factor and $ID$-factor-critical graphs, Discrete Math. 347 (2024) 113976.

\bibitem{FL} D. Fan, H. Lin, Binding number, $k$-factor and spectral radius of graphs, Electron. J. Combin. 31 (2024) $\#$P1.30.

\bibitem{FLL} D. Fan, H. Lin, H. Lu, Spectral radius and $[a, b]$-factors in graphs, Discrete Math. 345 (7) (2022) 112892.

\bibitem{FLLO} D. Fan, H. Lin, H. Lu, S. O, Eigenvalues and factors: a survey, Linear Algebra Appl. 743 (2026) 1-55.

\bibitem{FLZ} D. Fan, H. Lin, Y. Zhu, A note on the spectral radius and $[a, b]$-factor of graphs, Discrete
Math. 349 (2026) 115103.

\bibitem{G} X. Gu, Regular factors and eigenvalues of regular graphs, European J. Combin. 42 (2014) 15–25.

\bibitem{HL} Y. Hao, S. Li, Tur\'{a}n-Type problems on $[a,b]$-factors of graphs, and beyond, Electron. J. Combin. 32 (2024) $\#$P3.23.

\bibitem{H} Y. Hong, A bound on the spectral radius of graphs. Linear Algebra Appl. 108 (1988) 135–139.

\bibitem{HSF} Y. Hong, J. Shu, K. Fang, A sharp upper bound of the spectral radius of graphs, J. Combin. Theory, Ser. B 81 (2001) 177–183.

\bibitem{JFL} H. Jia, A. Fan, R. Liu, Spectral radius, the matching number and
fractional criticality of graphs, Linear Algebra Appl. 732 (2026) 1-17.

\bibitem{LFZ} Y. Li, D. Fan, Y. Zhu, Spectral radius and fractional $[a, b]$-factor of graphs, Linear Algebra Appl. 715 (2025) 32–45.

\bibitem{LLT} H. Liu, M. Lu, F. Tian, On the spectral radius of graphs with cut edges, Linear Algebra Appl. 389
(2004) 139–145.

\bibitem{L2} H. Lu, Regular graphs, eigenvalues and regular factors, J. Graph Theory 69 (4) (2012) 349–355.

\bibitem{MW} T. Ma, L. Wang, Distance spectral conditions for $ID$-factor-criticality and fractional
$[a, b]$-factor of graphs, Discrete Math. 349 (2026) 114803.

\bibitem{N} V. Nikiforov, Some inequalities for the largest eigenvalue of a graph, Combin. Probab. Comput. 11 (2)
(2002) 179–189.

\bibitem{O} S. O, Spectral radius and fractional matchings in graphs, European J. Combin. 55 (2016)
144–148.

\bibitem{O3} S. O, Spectral radius and matchings in graphs, Linear Algebra Appl. 614 (2021) 316-324.

\bibitem{O2} S. O, Eigenvalues and $[a, b]$-factors in regular graphs, J. Graph Theory 100 (3) (2022) 458–469.

\bibitem{TFZ} B.S. Tam, Y. Fan, J. Zhou, Unoriented Laplacian maximizing graphs are degree maximal, Linear Algebra Appl. 429 (2008) 735-758.

\bibitem{YL} Q. Yu, G. Liu, Graph Factors and Matching Extensions, Springer, 2009.

\bibitem{YH} Y. Yuan, H.X. Hao, A neighborhood union condition for fractional $ID$-$[a, b]$-factor-critical graphs,
Acta Math. Appl. Sin. Engl. Ser. 34 (2018) 775-781.

\bibitem{ZSL} S. Zhou, Z. Sun, H. Liu, A minimum degree condition for fractional $ID$-$[a, b]$-factor critical
graphs, Bull. Aust. Math. Soc. 86 (2012) 177-183.

\bibitem{ZY} S. Zhou, F. Yang, A neighborhood condition for fractional $ID$-$[a, b]$-factor-critical graphs, Discuss.
Math., Graph Theory 58 (2016) 409-418.

\end{thebibliography}
\end{document}